\documentclass[]{amsart}
\usepackage{amsmath, amsthm, amsfonts, amssymb}
\usepackage{mathrsfs}
\usepackage{microtype}
\usepackage[utf8]{inputenc}
\providecommand{\noopsort[1]{}}
\usepackage{bbm}
\usepackage{dsfont}
\usepackage{ifthen}
\usepackage{nicefrac}
\numberwithin{equation}{section}
\usepackage{a4}
\usepackage[active]{srcltx}
\usepackage{verbatim}
\usepackage{hyperref}
\usepackage{color}
\allowdisplaybreaks

\newtheorem{thm}{Theorem}[section]

\newtheorem{prop}[thm]{Proposition}
\newtheorem{lm}[thm]{Lemma}

\theoremstyle{remark}
\newtheorem{rmk}[thm]{Remark}
\newtheorem{rmks}[thm]{Remarks}

\newcommand{\coloneqq}{\mathrel{\mathop:}=}

\renewcommand{\Re}{{\rm Re}\,}

\renewcommand{\Im}{{\rm Im}\,}

\newcommand{\sg}{\mathrm{sign}}
\newcommand{\loc}{\mathrm{loc}}
\newcommand{\p}{\mathcal{P}}
\newcommand{\A}{\mathrm{a}}
\newcommand{\R}{\mathds{R}}
\newcommand{\C}{\mathds{C}}
\newcommand{\K}{\mathds{K}}
\newcommand{\N}{\mathds{N}}

\newcommand{\tnorm}[1]{{|\kern-0.3ex|\kern-0.3ex|#1|\kern-0.3ex|\kern-0.3ex|}}

\begin{document}
\title[Vector--valued elliptic operators]{On the $L^p$--theory of vector--valued elliptic operators}
%\author{M. Kunze}
%\address{Universit\"at Konstanz, Fachbereich Mathematik und Statistik, 78457 Konstanz, Germany}
%\email{markus.kunze@uni-konstanz.de}
\author{K. KHALIL}
\address{Department of Mathematics, Faculty of Sciences Semlalia, Cadi Ayyad University, B.P. 2390-40000, Marrakesh-Morocco.}
\email{kamal.khalil.00@gmail.com}

\author {A. MAICHINE}
%\thanks{Corresponding author: A. MAICHINE}
%\address{Dipartimento di Matematica, Universit\`a degli Studi di Salerno, Via Giovanni Paolo II 132, I-84084 Fisciano (SA), Italy}
\address{Mohammed VI Polytechnic University, Lot 660, Hay Moulay Rachid Ben Guerir, 43150, Morocco}
\email{abdallah.maichine@um6p.ma}
\date{}
%\footnote{Corresponding author.}
\begin{abstract}
In this paper, we study vector--valued elliptic operators of the form $\mathcal{L}f:=\mathrm{div}(Q\nabla f)-F\cdot\nabla f+\mathrm{div}(Cf)-Vf$ acting on vector--valued functions $f:\R^d\to\R^m$ and involving coupling at zero and first order terms.
We prove that  $\mathcal{L}$ admits realizations in $L^p(\R^d,\R^m)$, for $1<p<\infty$, that generate analytic strongly continuous semigroups provided that $V=(v_{ij})_{1\le i,j\le m}$ is a matrix potential with locally integrable entries satisfying a sectoriality condition, the diffusion matrix $Q$ is symmetric and uniformly elliptic and the drift coefficients $F=(F_{ij})_{1\le i,j\le m}$ and $C=(C_{ij})_{1\le i,j\le m}$ are such that $F_{ij},C_{ij}:\R^d\to\R^d$ are bounded.
 We also establish a result of local elliptic regularity for the operator $\mathcal{L}$, we investigate on the $L^p$-maximal domain of $\mathcal{L}$ and we characterize the positivity of the associated semigroup.
\end{abstract}
\maketitle

\section{Introduction}
The present paper deals with a class of vector--valued elliptic operators of the form
\begin{equation}
\mathcal{L}f=\mathrm{div}(Q\nabla f)-F\cdot\nabla f+\mathrm{div}(Cf)-Vf  \label{def of L}
\end{equation}
 acting on smooth functions $f:\R^d\to\R^m$, for some integers $ d,m \geq 1 $, and involving coupling through the first and zero order terms.
 More precisely, for $f=(f_1,\dots,f_m):\R^d\to\R^m$, one has
 \begin{equation*}
  (\mathcal{L}f)_i=\mathrm{div}(Q\nabla f_i)-\sum_{j=1}^{m}F_{ij}\cdot\nabla f_j+\sum_{j=1}^{m}\mathrm{div}(f_j C_{ij})-\sum_{j=1}^{m}v_{ij}f_j 
 \end{equation*}
 for each $i\in\{1,\dots,m\}$.\\
 
 We point out that the operator $\mathcal{L}$ appears in the study of Navier-Stokes equations. More precisely, in \cite{Triebel13,Triebel15}, H. Triebel used a reduced form of Navier--Stokes type equations on $\R^n$ (where $d=m=n $ in such case) that matches vector--valued semilinear parabolic evolution equations via the Leray/Helmoltz projector, see \cite[Chapter 6]{Triebel13} for details. %Our operator $ \mathcal{L} $ corresponds to the linear part of such parabolic evolution equations.
Moreover, a similar reduction method were applied in \cite{HieberRha,hieber} to convert Navier-Stokes equation to a semilinear parabolic system. The linear operator in \cite{HieberRha,hieber} is more appropriate to our situation.   Besides, parabolic systems appear also %in which the linear part is an operator of form $\mathcal{L}$, thereby the necessity of associating a semigroup to $\mathcal{L}$. 
 in the study of Nash equilibrium for stochastic differential games, see \cite{Fried71,Fried72,Mann04} and \cite[Section 6]{aalt}.\\
 
 %We assume that the potential matrix $V:=(v_{ij})_{1\le i,j\le d}$, satisfies the weakest condition that for each $ 1\leq i,j \leq d$, $ v_{ij}:\R^{d}\to\R$ is locally integrable and satisfying a sectoriality condition, i.e., \eqref{C-S for V(x)}. The matrix diffusion is given by $ Q=(q_{ij})_{1\le i,j\le d} $,  where $q_{ij}:\R^{d}\to\R$,  is symmetric and uniformly elliptic  and the matrix functions with vector--valued coefficients $F=(F_{ij})_{1\le i,j\le m}$ and $C=(C_{ij})_{1\le i,j\le m}$ are assumed to be bounded.\\
 
%Recently, due to its application in the study of fundamental properties of linear parabolic systems, the $L^p$-theory for elliptic operators has undergone a considerable progress from the scalar to the vector valued case which become a new tendency. 
In the scalar case, the theory of elliptic operators, is by now well understood, see \cite{Ouha-book}  and  \cite{Luca} for bounded  and unbounded coefficients respectively. However, the situation is quite different in the vector--valued case. Indeed, the interest into operators as in \eqref{def of L} in the whole space with possibly unbounded coefficients has started only in 2009 by Hieber et al. \cite{hetal09} with coupling through the lower order term of the elliptic operator and the motivation were the Navier-Stokes equation. Afterwards, few papers appeared, see \cite{aalt,ALP,DL11,KLMR,KMR,Mai,Mai-Rha}. In \cite{aalt,ALP,DL11} the authors studied the associated parabolic equation in $C_b$-spaces, assuming, among others, that the coefficients of the elliptic operator are H\"older continuous. In \cite{DL11}, solution to the parabolic system has been extrapolated to the $L^p$-scale provided the uniqueness.\\

In what concerns a Schr\"odinger type operator $\mathcal{A}=\mathrm{div}(Q\nabla\cdot)-V$, which corresponds to $F=C=0$ in \eqref{def of L}, and its associated semigroup, a comprehensive study in $L^p$-spaces can be find in \cite{KLMR,KMR,Mai,Mai-Rha}. Indeed, in \cite{Mai}, it has been associated a sesquilinear form to $\mathcal{A}$, for symmetric potential $V$, and it has been established a consistent $C_0$-semigroup in $L^p(\R^d,\R^m)$, $p\ge1$, which is analytic for $p\neq1$. This is done by assuming that $V$ is pointwisely semi-definite positive with locally integrable entries and $Q$ is symmetric, bounded and satisfies the well-known ellipticity condition. Moreover, the author investigated on compactness and positivity of the semigroup. In \cite{KMR}, the authors associated a $C_0$-semigroup, in $L^p$-spaces, which is not necessarily analytic, to the Schr\"odinger operator with typically nonsymmetric potential, provided that the diffusion matrix $Q$ is, in addition to the ellipticity condition, differentiable, bounded together with its first derivatives, $V$ is semi--definite positive and its entries are locally bounded. Here, the authors followed the approach adopted by Kato in \cite{Kato} for scalar Schr\"odinger operators with complex potential. The main tool has been local elliptic regularity and a Kato's type inequality for vector--valued functions, i.e.,
\[ \Delta_Q |f|\ge \frac{1}{|f|}\sum_{j=1}^{m}f_j\Delta_Q f_j \chi_{\{ f\neq 0 \} },
\]
for smooth functions $f:\R^d\to\R^m$, where $\Delta_Q:=\mathrm{div}(Q\nabla\cdot)$, see \cite[Proposition 2.3]{KMR}. Further properties such as maximal domain and others have been also investigated. The papers \cite{KLMR,Mai-Rha} focused on the domain of the operator and further regularity properties. So that, under growth and smoothness assumptions on $V$, the authors coincide the domain of $\mathcal{A}$ with its natural domain $W^{2,p}(\R^d,\R^m)\cap D(V_p)$, for $p\in(1,\infty)$, where $D(V_p)$  refers to the domain of multiplication by $V$ in $L^p(\R^d,\R^m)$. Furthermore, ultracontractivity, kernel estimates and, in the case of a symmetric potential, asymptotic behavior of the eigenvalues have been considered in \cite{Mai-Rha}.\\

In this article, using form methods and extrapolation techniques, we give a general framework of existence of analytic strongly continuous semigroup $\{S_p(t)\}_{t\geq 0}$ associated to suitable realizations of $\mathcal{L}$ in $L^p$-spaces, for $1<p < \infty$, under mild assumptions on the coefficients of $\mathcal{L}$. Namely, we assume that $Q$ is bounded and elliptic, $F$ and $C$ are bounded with a semi--boundedness condition on their divergences and $V$ has locally integrable entries and satisfies the following pointwise sectoriality condition
\begin{equation*}
|\Im\langle V(x)\xi,\xi\rangle|\le M\,\Re\langle V(x)\xi,\xi\rangle,
\end{equation*}
for all $x\in\R^d$ and all $\xi\in\C^m$. For further regularity, we assume that the entries of $Q$ are in $C_b^1(\R^d)$ and $V$ is locally bounded.  Note that, in \cite[Proposition 5.4]{KMR}, see also \cite[Proposition 4.5]{Mai-Rha}, the above inequality has been stated as a sufficient condition for the analyticity of the semigroup generated by realizations of $\mathcal{A}$ in $ L^p(\R^d,\R^m)$, $p\in(1,\infty)$. Moreover, by \cite[Example 4.3]{KLMR}, one can see that without such a condition one may not have an analytic semigroup. Note also that, even in the scalar case, the existence of a semigroup in $L^p$-spaces associated to elliptic operators with unbounded drift and/or diffusion terms is not a general fact, see \cite{PRS06} and \cite[Propostion 3.4 and Proposition 3.5]{MS12}. Furthermore, we point out that coupling through the diffusion (second order) term does not lead to $L^p$-contractive semigroups, see \cite{Cialdea}.\\

On the other hand, we establish a result of local elliptic regularity for solutions to elliptic systems, see Theorem \ref{thm ellip reg}. Namely, for given two vector--valued locally $p$--integrable functions $f ,g \in L^p_\loc(\R^d,\R^m)$ satisfying $\mathcal{L}f=g$ in a weak sense (distribution sense). Then $ f $ belongs to $W^{2,p}_{loc}(\R^d,\R^m)$, for $p\in(1,\infty)$. This result generalizes \cite[Theorem 7.1]{Agmon} to the vector--valued case. Thanks to this result we prove that the domain $D(L_p )$ of $L_p$, for $p\in(1,\infty)$, coincides with the maximal domain : 
\begin{equation*}
D_{p,\max}(\mathcal{L}):=\{f\in L^p(\R^d,\R^m)\cap W^{2,p}_{\loc}(\R^d;\R^m) : \mathcal{L}f\in L^p(\R^d;\R^m)\}.
\end{equation*}
We also characterize the positivity of the semigroup $\{S_p(t)\}_{t\geq 0}$. We prove that $\{S_p(t)\}_{t\geq 0}$ is positive if, and only if, the operator $\mathcal{L}$ is coupled only through the potential term and the coupling coefficients $v_{ij}$, $i\neq j$, are negative or null. \\
%generated by the realization of $\mathcal{L}$ in $L^p(\R^d,\R^m)$. We use  a general Beurling-Deny type criterion of invariance of closed convex subsets by semigroups, see Section \ref{sect 5}.

%\textcolor{red}{The organization of this paper is as follows, in Section \ref{sect 2}, we prove a generation result of a strongly continuous holomorphic semigroup $ \{ S_{2}(t) \}_{t\ge 0} $ to the realization of the operator $ \mathcal{L} $ in $ L^{2}(\R^{d}, \C^{m}) $, by using form method. In Section \ref{sect 3}, we extend the generation result obtained in Section \ref{sect 2} to  $ L^{p} $-spaces for all $ 1<p<\infty $, we use a Beurling-Deny type criterion to prove that $ \mathcal{L} $ generates an analytic strongly continuous semigroup $ \{ S_{p}(t) \}_{t\ge 0} $ in $ L^{p}(\R^{d}, \C^{m})  $. Section \ref{sec. maximal domain} is devoted to a characterization to the domain $D(L_p)$ of  the generator $ L_{p} $ of the semigroup $\{S_p(t)\}_{t\geq 0}$ for $ 1<p<\infty $, we prove that it is the maximal domain in $L^p(\R^d,\R^m)$. Technically, a local elliptic regularity result is needed. In fact, Theorem \ref{thm ellip reg} deal with a local elliptic regularity result for vector valued functions. We conclude this paper by Section \ref{sect 5} that concerns a characterization of positivity of the semigroup $\{S_{p}(t)\}_{t\ge0}$ for $ 1<p<\infty $.} \\

The organization of this paper is as follows: in Section \ref{sect 2}, we associate a sesquilinear form to the operator $\mathcal{L}$ in $L^2(\R^d,\mathds{C}^m)$ and we deduce the existence of an analytic $C_0$--semigroup $ \{ S_{2}(t) \}_{t\ge 0} $ associated to $\mathcal{L}$. In Section \ref{sect 3}, we prove that $ \{ S_{2}(t) \}_{t\ge 0} $ is quasi $L^\infty$--contractive and we extend $ \{ S_{2}(t) \}_{t\ge 0} $ to an analytic $C_0$--semigroup in $L^p(\R^d,\C^m)$ by extrapolation techniques. In Section \ref{sec. maximal domain}, we establish a local elliptic regularity result and we show that the domain of the generator of $ \{ S_{2}(t) \}_{t\ge 0} $ coincides with the maximal domain of $\mathcal{L}$ in $L^p(\R^d,\R^m)$, for $p\in(1,\infty)$. Section \ref{sect 5} is devoted to determine the positivity of $ \{ S_{2}(t) \}_{t\ge 0} $.

\subsection*{Notation} Let  $\K$ denotes the fields   $\R$ or $\C$, $d,m \geq 1$ any integers, $\langle\cdot,\cdot\rangle$ the inner-product of $\K^N$, $N=d,m$. So that, for $x=(x_1,\dots,x_N)$, $y=(y_1,\dots,y_N)$ in $\R^N$, $\langle x,y\rangle=\displaystyle\sum_{i=1}^{N}x_i \bar{y}_i$ and $x\cdot y=\displaystyle\sum_{i=1}^{N}x_i y_i$.

The space $L^p(\R^d,\K^m)$, $1<p<\infty$, is the vector--valued Lebesgue space endowed with the norm $$\|\cdot\|_p: f=(f_1,\dots,f_m)\mapsto\|f\|_p\coloneqq \left(\int_{\R^d}(\sum_{j=1}^{m}|f_j|^2)^{\frac{p}{2}}dx\right)^{\frac{1}{p}}.$$
We denote by $\langle\cdot,\cdot\rangle_{p,p'}$ the duality product between $L^p(\R^d,\K^m)$ and $L^{p'}(\R^d,\K^m)$ for $1<p<\infty$ where $p' =\frac{p}{p-1}$. For $p=2$, we denote it simply by $\langle\cdot,\cdot\rangle_2$.\\ We write $f\in L^p_\loc(\R^d,\K^m)$ if $\chi_{B}f$ belongs to $L^p(\R^d,\K^m)$ for every bounded $B\subset\R^d$, with $\chi_{B}$ is the indicator function of $B$.\\
For $k\in\N$, $W^{k,p}(\R^d ,\K^m)$ denotes the vector--valued Sobolev space constituted of vector--valued functions $f=(f_1,\dots,f_m)$ such that $f_j\in W^{k,p}(\R^d )$, for all $j\in\{1,\dots,m\}$, where $W^{k,p}(\R^d )$ is the classical Sobolev space of order $k$ over $L^p(\R^d)$. Note that all the derivatives are considered in the distribution sense. $W_{loc}^{k,p}(\R^d ,\K^m)$ is the set of all measurable functions $f$ such that the distributional derivative $\partial^\alpha f$ belongs to $L_{loc}^p(\R^d ,\K^m)$, for all $\alpha\in\N^d$ such that $|\alpha|\le k$.  For $y=(y_1,\dots,y_m)\in\R^m$, we write $y\ge 0$ if $y_j\ge 0$ for all $j\in\{1,\dots,m\}$. %For $r>0$, $B(r)=\{x\in\R^d : |x|<1 \}$ denotes the euclidean open ball of $\R^d$ of centre $0$ and radius $r$. For all set $E$, the  characteristic function $ \chi_{E} $ is given by: 
%\begin{equation*}
 %\chi_{E}(x) =\left\{
  %  \begin{aligned}
  % 1 & \quad \text{ if  } x\in E\\
  %0  &\quad  \text{ if not  }.  
   % \end{aligned} 
   % \right. 
%\end{equation*}

\section{The sesquilinear form and the semigroup in $L^2(\R^d,\C^m)$}\label{sect 2}
We consider the following differential expression
\begin{equation}\label{def 1 of L}
\mathcal{L}f=\mathrm{div}(Q\nabla f)-F\cdot\nabla f+\mathrm{div}(Cf)-Vf,
\end{equation}
where $f:\R^d\to\R^m$ and the derivatives are considered in the sense of distributions. Here, $Q=(q_{ij})_{1\le i,j\le d}$ and $V=(v_{ij})_{1\le i,j\le m}$ are matrices where the entries are scalar functions: $v_{ij},q_{ij}:\R^d\to\R$, and $F=(F_{ij})_{1\le i,j\le m}$ and $C=(C_{ij})_{1\le i,j\le m}$ are matrix functions with vector--valued entries: $F_{ij}, C_{ij}:\R^d\to\R^d$. So that $$(\mathrm{div}(Q\nabla f))_i=\mathrm{div}(Q\nabla f_i),$$ 
$$(F\cdot\nabla f)_i=\sum_{j=1}^{m}\langle F_{ij},\nabla f_j \rangle $$
$$(\mathrm{div}(Cf))_i=\sum_{j=1}^{m}\mathrm{div}(f_j C_{ij})$$ and $$  (Vf)_i = \sum_{j=1}^{m} v_{ij} f_j  $$  
for each $i\in\{1,\dots,m\}$.\\
Actually, for $f=(f_1,\dots,f_m)\in W^{1,p}_\loc(\R^d,\C^m)$ for some $1<p<\infty$, $\mathrm{div}(Q\nabla f)$, $F\cdot\nabla f$ and $\mathrm{div}(Cf)$ are vector--valued distributions and are defined as follow
\[(\mathrm{div}(Q\nabla f),\phi)=-\int_{\R^d}\sum_{i=1}^{m}\langle Q\nabla f_i ,\nabla \phi_i \rangle\, dx,\]
\[(F\cdot\nabla f,\phi)=\sum_{j=1}^{m}\int_{\R^d} (F_{ij}\cdot\nabla f_j)\bar{\phi}_i\, dx ,\]
and 
\[(\mathrm{div}(Cf),\phi)=-\sum_{j=1}^{m}\int_{\R^d}f_j \langle C_{ij},\nabla \phi_i\rangle\, dx\]
for every $\phi=(\phi_1,\dots,\phi_m)\in C_c^\infty(\R^d,\C^m)$.\\

Throughout this paper we make the following assumptions\\
\textbf{Hypotheses (H1)}:
%\begin{hyp}\label{hyp H1}
\begin{itemize}
\item $Q:\R^d\to\R^{d\times d}$ is measurable such that, for every $x\in\R^d$, $Q(x)$ is symmetric and there exist $\eta_1,\eta_2>0$ such that
\begin{equation}
\eta_1|\xi|^2\le\langle Q(x)\xi,\xi\rangle\le \eta_2|\xi|^2,
\end{equation}
for all $x,\xi\in\R^d$.\\
\item $F_{ij}, C_{ij}\in L^\infty(\R^d,\R^d)$, for all $i,j\in\{1,\dots,d\}$.\\
%\item $\left(\mathrm{div}(Cu)\right)_i=\sum_{j=1}^{m}\mathrm{div}(u_j C_{ij})$ and $C_{ij}\in L^\infty(\R^d,\R^d)$.
\item $v_{ij}\in L^1_\loc (\R^d)$, for every $i\in\{1,\dots,m\}$ and there exists $M>0$ such that
\begin{equation}\label{cond on V}
|\Im\langle V(x)\xi,\xi\rangle|\le M\,\Re\langle V(x)\xi,\xi\rangle,
\end{equation}
for all $x\in\R^d$ and all $\xi\in\C^m$.
\end{itemize}
%\end{hyp}
Let us define, for every $x\in\R^d$, $V_s(x):=\frac{1}{2}(V(x)+V^*(x))$ to be the symmetric part of $V(x)$, where $V^*(x)$ is the conjugate matrix of $V(x)$. $V_{as}(x):=V(x)-V_s(x)$ denotes the antisymmetric part of $V(x)$.

We start by a technical lemma
\begin{lm}\label{lm sect of V}
Let $x\in\R^d$ and assume $V$ satisfying \eqref{cond on V}. Then

\begin{equation}\label{C-S for V(x)}
|\langle V(x)\xi_1,\xi_2\rangle|\le (1+M) \langle V_s(x)\xi_1,\xi_1\rangle^{1/2}\langle V_s(x) \xi_2,\xi_2\rangle^{1/2}
\end{equation}
for every $\xi_1,\xi_2\in\C^m$. Moreover, the inequality holds true also when substituting $V$ by $V_{as}$.\\
In particular,
\begin{equation}\label{C-S for V in L^2}
\left\vert\int_{\R^d}\langle V_{as}(x)f(x),{g}(x)\rangle\,dx\right\vert \le (1+M)\|V_s^{1/2}f\|_{L^2(\R^d,\C^m)}\| V_s^{1/2} g\|_{L^2(\R^d,\C^m)}.
\end{equation}
for every measurable $f$ and $g$ such that $V_s^{1/2} f,\, V_s^{1/2} g\in L^2(\R^d,\C^m)$.
\end{lm}
\begin{proof}
For $x\in\R^d$, $\langle V(x)\cdot,\cdot\rangle$ is a sesquilinear form over $\C^m$. Taking into the account that, for every $\xi\in\C^m$, $ \Re\langle V(x)\xi,\xi\rangle=\langle V_s(x)\xi,\xi\rangle .$
Then, \eqref{C-S for V(x)} follows by \eqref{cond on V} and \cite[Proposition 1.8]{Ouha-book}. Moreover, \eqref{C-S for V(x)} holds true also when taking $V_{as}$ instead of $V$ in the left hand side of the inequality. Now, Cauchy Schwartz inequality yields \eqref{C-S for V in L^2}. 
\end{proof}

Let us now consider the sesquilinear form $\A$ given by
\begin{align*}
\A(f,g) &:=\sum_{i=1}^{m}\int_{\R^d}\langle Q\nabla f_i ,\nabla g_i \rangle\, dx+\sum_{i,j=1}^{m}\int_{\R^d} (F_{ij}\cdot\nabla f_j) \bar{g}_i\, dx\\
&+\sum_{i,j=1}^{m}\int_{\R^d}f_j \langle C_{ij},\nabla g_i\rangle \, dx+\int_{\R^d}\langle V f,g\rangle\, dx,
\end{align*}
with domain
\[ D(\A)=\{f\in H^1(\R^d,\C^m): \int_{\R^d}\langle V_s f,f\rangle dx<\infty\}:=D(\A_0),
\]
where
\[ \A_0(f,g)=\sum_{j=1}^{m}\langle Q\nabla f_j,\nabla{g}_j\rangle_{2}+\int_{\R^d}\langle V_s(x)f(x),{g}(x)\rangle\, dx .\] 
%and $V_s=\frac{1}{2}(V+V^*)$.
%If $V_{as}=\frac{1}{2}(V-V^*)$, then
%\[ \langle V_s(x) \xi,\xi\rangle=\Re \langle V(x) \xi,\xi\rangle,\qquad x\in\R^d,\xi\in\C^m,\]
%and
%\[ \langle V_{as}(x) \xi,\xi\rangle=i \Im \langle V(x) \xi,\xi\rangle,\qquad x\in\R^d,\xi\in\C^m.\]
%We split $\A$ into $\A=\A_0+\mathrm{b}_1+\mathrm{b_2}$, where
%\[\mathrm{b}_1(f,g):=\sum_{i,j=1}^{m}\int_{\R^d}\left(F_{ij}\cdot\nabla f_j\right) \bar{g}_i\, dx+\sum_{i,j=1}^{m}\int_{\R^d}f_j \left(C_{ij}\cdot\nabla \bar{g}_i\right)\, dx\]
%and \[\mathrm{b_2}(f,g):=\int_{\R^d}\langle V_{as}(x)f(x),{g}(x)\rangle\, dx.\]
%One has
%\[|\mathrm{b}_1(f,g)|\le C_1(\|\nabla f\|_2\|g\|_2+\|f\|_2\|\nabla g\|_2)\]
%and \[\mathrm{b}_2(f,g)\le C_2\langle V_s f,f\rangle_{L^2(\R^d,\C^m)}\langle V_s g,g\rangle_{L^2(\R^d,\C^m)}.\]
%Consequently, there exists $\omega\ge0$ such that
 The form $\A$ satisfies the following properties 
\begin{prop}\label{prop properties of a}
Assume Hypotheses (H1) are satisfied. Then,
\begin{itemize}
\item $\A$ is densely defined; 
\item there exists $\omega>0$ such that $\A_\omega:=\A+\omega$ is accretive: $\Re \A(f)+\omega\|f\|_2^2\ge0$, for all $f\in D(\A)$;
\item $\A$ is continuous;
\item $\A$ is closed on $D(\A)$. %endowed with $\|\cdot\|_\A:=\left((1+\omega)\|\cdot\|_2+\Re\A\right)^{1/2}\sim \|\cdot\|_{\A_0}$.
\end{itemize}
\end{prop}
\begin{proof}
Clearly, $C_c^\infty(\R^d,\C^m)\subseteq D(\A)$ and thus, $\A$ is densely defined. Moreover, by application of Young's inequality, one obtains, for every $f\in D(\A)$ and every $\varepsilon>0$,
\begin{align*}
\Re\A(f) &=\sum_{i=1}^{m}\int_{\R^d}|\nabla f_i|_Q^2\, dx+\sum_{i,j=1}^{m}\int_{\R^d}\Re\left( (F_{ij}\cdot\nabla f_j)\bar{f}_i\right) \, dx\\
& \quad+\sum_{i,j=1}^{m}\int_{\R^d}\Re\left( f_j \langle C_{ij} ,\nabla f_i\rangle\right) \, dx+\int_{\R^d}\Re\langle V f,f\rangle\, dx\\
& \ge \eta_1\sum_{i=1}^{m}\int_{\R^d}|\nabla f_i|^2\,dx -(\|F\|_{\infty}+\|C\|_{\infty})\int_{\R^d}\sum_{i=1}^{m}|f_i|\sum_{i=1}^{m}|\nabla f_i|\,dx\\
& \ge (\eta_1-\varepsilon)\int_{\R^d}\sum_{i=1}^{m}|\nabla f_i|^2\,dx- c_\varepsilon \int_{\R^d}\sum_{i=1}^{m}|f_i|^2\,dx.
\end{align*}
So by choosing $\varepsilon=\eta_1/2$ and $\omega\ge c_{\eta_1/2}$, one obtains $\Re\A(f)+\omega\|f\|_2^2\ge0$, which shows that $\A_\omega$ is accretive.\\
On the other hand, according to \cite[Proposition 2.1]{Mai}, $(D(\A),\|\cdot\|_{\A_0})$ is a Banach space, where
$$ \|\cdot\|_{\A_0}:=\sqrt{\|\cdot\|_{2}^2+\A_0(\cdot)}.$$
It is then enough to show that $\|\cdot\|_\A$ is equivalent to $\|\cdot\|_{\A_0}$ to conclude the closedness of $\A$, where $\|\cdot\|_\A$ is the graph norm associated to $\A$ and it is given by $$\|\cdot\|_\A:=\sqrt{(1+\omega)\|\cdot\|_2^2+\Re a(\cdot)}.$$
Here $\omega$ is such that $\A_\omega$ is accretive. Let us first prove that $\|\cdot\|_{\A}\lesssim\|\cdot\|_{\A_0}$. Let $f\in D(A)$, one has $ \\A(f)=\A_0(f)+b(f)$, where
\[ \mathrm{b}(f):=\sum_{i,j=1}^{m}\int_{\R^d}(F_{ij}\cdot\nabla f_j) \bar{f}_i \, dx+\sum_{i,j=1}^{m}\int_{\R^d}f_j \langle C_{ij} ,\nabla f_i\rangle\, dx.
\]
The claim then follows by application of Young's inequality when estimating $\mathrm{b}$ as in the align above. Conversely, since $\A_0(f)= \A(f)-b(f)$, in a similar way one deduces that $\|\cdot\|_{\A_0}\lesssim\|\cdot\|_{\A}$.\\
It remains to show that $\A$ is continuous in $(D(A),\|\cdot\|_{\A_0})$, that is
\[ |\A(f,g)|\le c\|f\|_{\A_0}\|g\|_{\A_0},\qquad\forall f,g\in D(\A). \]
In view of \eqref{C-S for V in L^2}, Cauchy-Schwartz inequality and the continuity of $\A_0$, c.f. \cite[Proposition 2.1 (iii)]{Mai}, one gets
\begin{align*}
|\A(f,g)|&\le |\A_0(f,g)|+|\mathrm{b}(f,g)|+\left|\int_{\R^d}\langle V_{as}f,g\rangle\right|\\
&\le c_1\|f\|_{\A_0}\|g\|_{\A_0}+c_2\|f\|_{H^1(\R^d,\C^m)}\|g\|_{H^1(\R^d,\C^m)}+c_3 \|V_s^{1/2}f\|_{2}\| V_s^{1/2} g\|_{2}\\
&\le c\|f\|_{\A_0}\|g\|_{\A_0}.
\end{align*}
%In order to prove closedness and continuity of $\A$, we first show that $\|\cdot\|_\A$ is actually equivalent to $||\cdot||_{\A_0}$, where $$\|\cdot\|_\A:=\sqrt{(1+\omega)\|\cdot\|_2^2+\Re a(\cdot)}$$
%we prove the continuity and closedness of $\A$ on 
%and then we show that $\|\cdot\|_{\A_0}$ is equivalent to $\|\cdot\|_\A$, where 

\end{proof}

We, finally, conclude the main theorem of this section as an immediate consequence of \cite[Proposition 1.51 \and Theorem 1.52]{Ouha-book} and Proposition \ref{prop properties of a}
\begin{thm}\label{thm Generation L2}
Assume Hypotheses (H1) are satisfied. Then, $\mathcal{L}$ admits a realization $L=L_2$ in $L^2(\R^d,\C^m)$ that generates an analytic $C_0$-semigroup $\{S_{2}(t)\}_{t\ge0}$. Moreover, there exists $\omega\ge0$ such that
\[\|S_{2}(t)\|_{2}\le \exp(\omega t),\quad\mathrm{for\;every}\; t\ge0.\]
\end{thm}

\section{Extrapolation of the semigroup to the $L^p$--scale}\label{sect 3}
%The aim of this section is to prove that  the realization of the operator $ \mathcal{L} $ in $L^p(\mathbb{R}^{d}, \mathbb{C}^{m})$ for $ 1<p<\infty $, generates a holomorphic semigroup.
%In this section we show that,  for every $ t>0 $,  $\left[ \exp(-\omega t)S_{2}(t):= S_{2}^{\omega }(t)\right]  $ defined on $L^2(\mathbb{R}^{d}, \mathbb{C}^{m}) \cap L^\infty (\mathbb{R}^{d}, \mathbb{C}^{m}) $ can be extrapolated to a bounded linear operator $ S_{p}(t) $  in $ L^p(\mathbb{R}^{d}, \mathbb{C}^{m}) $ for $ 1<p<\infty $ by extrapolation for $ 2<p<\infty $ and by duality arguments for $ 1<p<2 $. 
In this section we extrapolate $\{S_2(t)\}_{t\ge0}$ to an analytic strongly continuous semigroup in $L^p(\R^d,\R^m)$. For that purpose, it suffices to prove that there exists $\tilde{\omega}\in\R$ such that $\{S^{\tilde{\omega}}_{2}(t):=\exp(-\tilde{\omega} t)S_{2}(t)\}_{t\ge 0}$ satisfies the following $L^\infty$-contractivity property:
\begin{equation}\label{L^8 contractivity}
\|S_{2}^{\tilde{\omega}}(t)f\|_\infty\le \|f\|_\infty,\qquad \forall f\in L^2(\R^d,\C^m)\cap L^\infty(\R^d,\C^m).
\end{equation}

From now on, we use the following notation: $$\langle y,z\rangle_{Q(x)}:=\langle Q(x)y,z\rangle$$ and $$|y|_{Q(x)}:=\sqrt{\langle Q(x)y,y\rangle},$$ for every $x,y,z$ in $\R^d$. We also drop the $x$ and denotes simply $\langle\cdot,\cdot\rangle_Q$ and $|\cdot|_Q$ for the ease of notation.

In this section we make use of the following hypotheses\\

\textbf{Hypotheses (H2)}:
%\begin{hyp}\label{hyp H2}
\begin{itemize}
%\item $q_{ij}\in C_b^1(\R^d)$, for all $i,j\in\{1,\dots,d\}$.
\item $F_{ij}, C_{ij}\in W^{1,\infty}_\loc(\R^d)$, for all $i,j\in\{1,\dots,m\}$, and there exists $\gamma\in\R$ such that
\begin{equation}\label{eq. div F>...}
\langle \mathrm{div}(F)(x)\xi,\xi\rangle:=\sum_{i,j=1}^{m}\mathrm{div}(F_{ij})(x)\xi_i\xi_j \le \gamma|\xi|^2 
\end{equation}
and
\begin{equation}\label{eq. div C>...}
\langle \mathrm{div}(C)(x)\xi,\xi\rangle:=\sum_{i,j=1}^{m}\mathrm{div}(C_{ij})(x)\xi_i\xi_j \le \gamma|\xi|^2 
\end{equation}
for every $\xi\in\R^m$ and $x\in\R^d$.
\end{itemize}

We state, now, the first result of this section
\begin{prop}\label{Charac Ouhab}
Assume Hypotheses (H1) and (H2). Then there exists $\tilde{\omega}\in\R$ such that $\{S^{\tilde{\omega}}_{2}(t):=\exp(-\tilde{\omega} t)S_{2}(t)\}_{t\ge 0}$  is $L^\infty$-contractive.
% Let $\{T(t)\}_{t\geq 0}$ be a semigroup of linear operators on $L^2(\mathbb{R}^{d}, \mathbb{R}^{m})$. By $L^\infty$--contractivity of $\{T(t)\}_{t\geq 0}$, we mean the property
\end{prop}
\begin{proof}
%We establish $L^\infty$--contractivity of $\{S_{2}^{\omega}(t)\}_{t\ge0}$ by applying a Beurling--Deny type criterion. 
According to the characterization of $L^\infty$-contractivity property given by \cite[Theorem 1]{Ouhabaz 99}, it suffices to prove that: for $\tilde{\omega}\ge0$ such that $\A_{\tilde{\omega}}$ is accretive, the following statements hold:
\begin{enumerate}
\item $f\in D(\A)$ implies $(1 \wedge|f|)\mathrm{sign}(f)\in D(\A)$,
\item $ \Re\A_{\tilde{\omega}}\left(f,f-(1\wedge |f|)\mathrm{sign}(f)\right)\ge 0,\qquad\forall f\in D(\A), $
\end{enumerate}
where $\mathrm{sign}(f):=\frac{f}{|f|}\chi_{\{ f\neq 0 \} }$. The first item follows by \cite[Lemma 3.2]{Mai}. Let us show (2). Set $\p_{f}:=(1\wedge |f|)\mathrm{sign}(f)$ and let $\tilde{\omega}$ be bigger enough, so that $\A_{\tilde{\omega}}$ is accretive and $\tilde{\omega}\ge\gamma$. According to \cite[Lemma 3.2]{Mai}, we claim that 
\begin{align}
\nabla(\p_{f})_i & =\dfrac{1+\sg(1-|f|)}{2}\frac{f_i}{|f|}\chi_{\{f\neq 0\}}\nabla|f|
+\dfrac{1\wedge|f|}{|f|}(\nabla f_i-\frac{f_i}{|f|}\nabla|f|)\chi_{\{f\neq 0\}}\\
&= \dfrac{1\wedge|f|}{|f|}\chi_{\{f\neq 0\}}\nabla f_i+\left(\dfrac{1+\sg(1-|f|)}{2}-\dfrac{1\wedge|f|}{|f|}\right)\frac{f_i}{|f|}\chi_{\{f\neq 0\}}\nabla|f|\nonumber
\end{align}
for every $i\in\{1,...,m\}$. Therefore,
\begin{align*}
\A_{\tilde{\omega}} (f,(f-\p_{f}))&:= \sum_{i=1}^{m}\int_{\R^d}\langle Q\nabla f_i ,\nabla (f-\p_{f})_i \rangle\, dx+\sum_{i,j=1}^{m}\int_{\R^d}\left(F_{ij}\cdot\nabla f_j\right) (\overline{f-\p_{f}})_i\, dx \\
  & +\sum_{i,j=1}^{m}\int_{\R^d}f_j \langle C_{ij},\nabla (f-\p_{f})\rangle\, dx+\int_{\R^d}\langle V f,(f-\p_{f})\rangle\, dx +\tilde{\omega} \langle f,(f-\p_{f})\rangle_2 \\
  &=  \tilde{\A}_0(f,f-\p_f )+\mathrm{b} (f,f-\p_f ) +\int_{\R^d}\langle V f,(f-\p_{f})\rangle\, dx+ \tilde{\omega} \langle f,(f-\p_{f})\rangle_2,
\end{align*}
where $$ \tilde{\A}_0(f,f-\p_f)=\displaystyle \sum_{i=1}^{m}\int_{\R^d}\langle Q\nabla f_i ,\nabla (f-\p_f)_i \rangle\, dx $$ 
and
$$  \mathrm{b}(f,f-\p_f )= \sum_{i,j=1}^{m}\int_{\R^d}\left(F_{ij}\cdot\nabla f_j\right) (\overline{f-\p_{f}})_i\, dx +\sum_{i,j=1}^{m}\int_{\R^d}f_j \langle C_{ij},\nabla (f-\p_{f})\rangle dx.$$

Now, one has
 \begin{eqnarray*}
\int_{\R^d}\langle V f,(f-\p_{f})\rangle\, dx = %\int_{\R^d}\langle V f,f\rangle\, dx - \int_{\R^d}\langle V f,\p_{f}\rangle\, dx \\ &=& \int_{\R^d}\langle V f,f\rangle\, dx - \int_{\R^d}\langle V f,(1\wedge |f|) \dfrac{f}{|f|}\chi_{\{ f\neq 0 \} } \rangle \, dx \\ &=& 
\int_{\R^d} \left(1- \dfrac{1\wedge |f|}{|f|}\chi_{\{ f\neq 0 \} } \right)  \langle V f,f\rangle\, dx .
\end{eqnarray*}
Consequently, since by \eqref{cond on V}, $\Re \langle  V f,f\rangle \geq 0 $ a.e., it follows that 
 \begin{equation} \label{VFormula1}
E_{1}:=\Re \int_{\R^d}\langle V f,(f-\p_{f})\rangle\, dx  \geq 0.
\end{equation}
%In addition, we have
%\begin{eqnarray}
%E_{2}:=\omega \Re \langle f,(f-\p_{f})\rangle = \omega \left(1-\dfrac{1\wedge |f|}{|f|} \chi_{\{ f\neq 0 \} }  \right) \|f\|^{2}_{2}.
%\end{eqnarray}
On the other hand, 
\begin{eqnarray}
\tilde{\A}_0(f,f-\p_f )&=&\tilde{\A}_0(f,f )- \tilde{\A}_0(f,\p_f ) \nonumber \\ &= &\sum_{i=1}^{m} \int_{\R^d}\underbrace{\left(1- \dfrac{1\wedge |f|}{|f|}\chi_{\{ f\neq 0 \} } \right)}_{\alpha(|f|)} \langle Q\nabla f_i ,  \nabla f_i \rangle\, dx \nonumber \\ & & +\sum_{i=1}^{m}  \int_{\R^d} \underbrace{\left[ \dfrac{1\wedge |f|}{|f|}\chi_{\{ f\neq 0 \} } -\left(\dfrac{1+\mathrm{sign}(1-|f|)}{2} \right)\chi_{\{ f\neq 0 \} }  \right]}_{\beta(|f|)} \langle Q\nabla f_i , \dfrac{f_i}{|f|} \nabla|f|  \rangle\, dx. \label{A0Formula1}
\end{eqnarray}

Applying an integration by part, one obtains  
%\begin{eqnarray}
%\A_1 (f,f-\p_f )&=&\A_1(f,f )-\A_1(f,\p_f )\nonumber \\ &= &\sum_{i,j=1}^{m}\int_{\R^d}\left(1-\dfrac{1\wedge |f|}{|f|} \chi_{\{ f\neq 0 \} }  \right) \left(F_{ij}\cdot\nabla f_j\right) f_i\, dx \nonumber \\ & &+\sum_{i,j=1}^{m}\int_{\R^d}   \left[\dfrac{1\wedge |f|}{|f|} \chi_{\{ f\neq 0 \} }-\left(\dfrac{1+\mathrm{sign}(1-|f|)}{2} \right)\chi_{\{ f\neq 0 \} }  \right] f_j \langle C_{ij},\nabla |f|\rangle \dfrac{f_i}{|f|} \, dx \nonumber  \\ & & +\sum_{i,j=1}^{m}  \int_{\R^d}  \left(1-\dfrac{1\wedge |f|}{|f|} \chi_{\{ f\neq 0 \} }  \right) f_j \langle C_{ij} , \nabla f_i \rangle  \, dx. \label{B1Formula1}
%\end{eqnarray}

\begin{align*}
\mathrm{b} (f,f-\p_f )&= \sum_{i,j=1}^{m}\int_{\R^d}\left(1-\dfrac{1\wedge |f|}{|f|} \chi_{\{ f\neq 0 \} }  \right) \left(F_{ij}\cdot\nabla f_j\right) \bar{f}_i\, dx
 +\sum_{i,j=1}^{m}\int_{\R^d}f_j \langle C_{ij},\nabla (f-\p_{f})\rangle dx.\\
 &= \sum_{i,j=1}^{m}\int_{\R^d}\left(1-\dfrac{1\wedge |f|}{|f|} \chi_{\{ f\neq 0 \} }  \right) \left(F_{ij}\cdot\nabla f_j\right) \bar{f}_i\, dx
 -\sum_{i,j=1}^{m}\int_{\R^d}\mathrm{div}(f_j C_{ij})(\overline{f-\p_f})_i dx\\
 &= \sum_{i,j=1}^{m}\int_{\R^d}\left(1-\dfrac{1\wedge |f|}{|f|} \chi_{\{ f\neq 0 \} }  \right) \left(F_{ij}\cdot\nabla f_j\right) \bar{f}_i\, dx\\
 & -\sum_{i,j=1}^{m}\int_{\R^d}(C_{ij}\cdot\nabla f_j)(\overline{f-\p_f})_i dx - \sum_{i,j=1}^{m}\int_{\R^d}\mathrm{div}(C_{ij})f_j (\overline{f-\p_f})_i dx\\
 & =  \sum_{i,j=1}^{m}\int_{\R^d}\left(1-\dfrac{1\wedge |f|}{|f|} \chi_{\{ f\neq 0 \} }  \right) \left((F_{ij}-C_{ij})\cdot\nabla f_j\right) \bar{f}_i\, dx - \langle \mathrm{div}(C^*)f,f-\p_f\rangle,
\end{align*}

where $\mathrm{div}(C^*):=(\mathrm{div}(C))^*$ is (pointwisely) the conjugate matrix of $\mathrm{div}(C)=(\mathrm{div}(C_{ij}))_{1\le i,j\le m}$.\\

Summing up one obtains
\begin{align*}
\Re \A_{\tilde{\omega}} (f,(f-\p_{f})) =& \Re \A_0 (f,(f-\p_{f}))+\Re \A_1 (f,f-\p_f )+\Re\int_{\R^d}\langle (V+\tilde{\omega} I_m)f,f-\p_f \rangle dx\\
=& \int_{\R^d}\alpha(|f|) \sum_{i=1}^{m} |\nabla f_i|_Q \;dx  +\int_{\R^d} \frac{\beta(|f|)}{|f|} \sum_{i=1}^{m}  \langle \Re(\bar{f_i}\nabla f_i) , \nabla|f| \rangle_Q \; dx\\
&+ \sum_{i,j=1}^{m}\int_{\R^d}\left(1-\dfrac{1\wedge |f|}{|f|} \chi_{\{ f\neq 0 \} }  \right) \Re\left[((F_{ij}-C_{ij})\cdot\nabla f_j)\bar{f}_i\right]  dx \\
 & + \Re\int_{\R^d}\langle(V-\mathrm{div}(C^*)+\tilde{\omega} I_m)f,f-\p_f \rangle dx\\
 =& \int_{\R^d}\alpha(|f|) J_{1}(f)  dx + \int_{\R^d}  \beta(|f|) J_{2}(f)\, dx\\
 & + \Re\int_{\R^d}\langle(V-\mathrm{div}(C^*)+\tilde{\omega} I_m)f,f-\p_f \rangle dx
\end{align*}
where 
\[ J_1(f):=  \sum_{i=1}^{m}  \Re \langle Q\nabla f_i ,  \nabla f_i \rangle + \sum_{i,j=1}^{m}\Re\left[ \left((F_{ij}-C_{ij})\cdot\nabla f_j\right) \bar{f}_i \right] 
\]
and 
\[ J_2(f):= \frac{1}{|f|}\sum_{i=1}^{m}  \langle \Re(\bar{f_i}\nabla f_i) , \nabla|f| \rangle_Q  .
\]
Since by \cite[Lemma 2.4]{Mai-Rha}, one has
\begin{equation*}\label{gradient of norm f}
\nabla |f|=\dfrac{\displaystyle\sum_{j=1}^{m}\Re (\bar{f}_j\nabla f_j)}{|f|}\chi_{\{f\neq 0\}}.
\end{equation*}
%and
%$$  |\nabla|f||^2\le    \sum_{i=1}^{m} | \nabla f_i |^{2}.$$
Then,
\begin{eqnarray*}
J_{2}(f) &=  &  \frac{1}{|f|}  \langle \sum_{i=1}^{m}\Re(\bar{f_i}\nabla f_i) , \nabla|f| \rangle_Q \\
 & = &   \langle  \nabla |f| ,  \nabla |f|   \rangle_Q \ge 0.
\end{eqnarray*}
Therefore,
\begin{equation}\label{eq J2>0}
\int_{\R^d}\beta(|f|)J_2(f) dx \ge0.
\end{equation}
%\begin{eqnarray*}
%\Re \A_\omega (f,(f-\p_{f}))&:=& \sum_{i=1}^{m}\int_{\R^d}\Re \langle Q\nabla f_i ,\nabla (f-\p_{f})_i \rangle\, dx+\sum_{i,j=1}^{m}\int_{\R^d}\Re \left(F_{ij}\cdot\nabla f_j\right) (f-\p_{f})_i\, dx \\ 
%& & +\sum_{i,j=1}^{m}\int_{\R^d}\Re f_j \langle C_{ij},\nabla (f-\p_{f})_i\rangle\, dx+\int_{\R^d}\Re \langle V f,(f-\p_{f})\rangle\, dx  \\
 %& & +\underbrace{ \omega \left\Vert \left(1-\dfrac{1\wedge |f|}{|f|} \chi_{\{ f\neq 0 \} }  \right)^{\frac{1}{2}} f\right\Vert ^{2}_{2}}_{E_2}.
%\end{eqnarray*}
Moreover, according to \eqref{eq. div C>...} that holds true also for $C^*$, one gets
\begin{align}\label{eq div C+omega>...}\nonumber
\Re\int_{\R^d}\langle(-\mathrm{div}(C^*)+\tilde{\omega} I_m)f,f-\p_f \rangle dx &= \int_{\R^d}\left(1-\dfrac{1\wedge |f|}{|f|} \chi_{\{ f\neq 0 \} }  \right)\langle (-\mathrm{div}(C^*)+\omega)f,f\rangle dx\\ \nonumber
& \ge (\tilde{\omega}-\gamma)\int_{\R^d}\left(1-\dfrac{1\wedge |f|}{|f|} \chi_{\{ f\neq 0 \} }  \right)|f|^2 dx\\
& = (\tilde{\omega}-\gamma)\int_{\R^d}\alpha( |f|)|f|^2 dx.
\end{align}

Now, taking in consideration \eqref{VFormula1}, \eqref{eq J2>0} and \eqref{eq div C+omega>...}, one obtains
\[
\Re \A_{\tilde{\omega}} (f,(f-\p_{f})) \ge  \int_{\R^d}\alpha(|f|) J_{1}(f) \, dx +
(\tilde{\omega}-\gamma)\int_{\R^d}\alpha( |f|)|f|^2 dx.
\]

Moreover, in view of Young's inequality, for every $\varepsilon>0$ there exists $c_\varepsilon>0$ such that 
\begin{eqnarray*}
J_{1}(f) &\geq & \eta_{1} \sum_{i=1}^{m} | \nabla  f_{i} |^{2} - \sum_{i,j=1}^{m}  |\langle (F_{ij}-C_{ij}),\nabla f_j\rangle|\, | f_i | \\
&\geq &   \eta_{1} \sum_{i=1}^{m} |   \nabla f_{i} |^{2} - \sup_{i,j}\|F_{ij}-C_{ij}\|_\infty\sum_{i,j=1}^{m} |  \nabla f_j|\,| f_i | \\
 &\geq &   \eta_{1} \sum_{i=1}^{m}| \nabla  f_{i} |^{2} - \varepsilon \sum_{i=1}^{m}  | \nabla  f_i|^{2}-c_{\varepsilon}\sum_{i=1}^{m}  | f_i | ^{2}
\\ &= &  ( \eta_{1}-\varepsilon) \sum_{i=1}^{m} | \nabla  f_{i} |^{2} - c_{\varepsilon} | f| ^{2}.\\
\end{eqnarray*}

%\\ &\geq  &   | \nabla |f| \ |_{Q}^{2}  - c\, (\sum_{i=1}^{m}| \nabla f_i |^2)^{1/2}
%\\ &\geq  &    | \nabla |f| \ |_{Q}^{2}  - \dfrac{\varepsilon}{2}  \sum_{i=1}^{m} | \nabla f_i |^{2} - \dfrac{c_{\varepsilon}}{2}\sum_{i=1}^{m}| f_i |^{2}
%\end{eqnarray*}
%for every $\varepsilon>0$. Note that the constant $c$ in the above align may change from line to line.
Consequently, for  $ \varepsilon$ being such that $ \eta_{1}>\varepsilon $, say $\varepsilon=\eta_1/2$, and $ \tilde{\omega} > c_{\eta_1/2}+\gamma  $, one gets
\begin{eqnarray*}
\Re \A_{\tilde{\omega}} (f,(f-\p_{f}))&\geq &  \int_{\R^d} \alpha(|f|)  \left[  ( \eta_{1}-\varepsilon) \sum_{i=1}^{m} | \nabla  f_{i} |^{2} + \left(\tilde{\omega}-\gamma- c_{\varepsilon}\right)|f| ^{2}\right]  \, dx \\ 
&\ge & 0
\end{eqnarray*}
and this ends the proof.
%& \geq  &   \int_{\R^d} M(|f|) \left[ ( \eta_{1}-\varepsilon) \sum_{i=1}^{m} | \nabla  f_{i} |^{2} - C''_{\varepsilon}\sum_{i=1}^{m}  | f_i | ^{2}\right]  \, dx + \\ 
%& +& \int_{\R^d}  \left[ \dfrac{1\wedge |f|}{|f|^{2}}\chi_{\{ f\neq 0 \} } -\left(\dfrac{1+\mathrm{sign}(1-|f|)}{2|f|} \right)\chi_{\{ f\neq 0 \} }  \right]  | f| \cdot | \nabla |f| \ |_{Q}^{2}   \, dx \\
%&+ & E_{1} + E_{2}\\
%& \geq  &  ( \eta_{1}-\varepsilon)  \| M(|f|)^{\frac{1}{2}}. \nabla  f \|_{2}^{2}+( \omega - C''_{\varepsilon}) + \| M(|f|)^{\frac{1}{2}}. f \|_{2}^{2}\\ 
%& +& \int_{\R^d}  \left[ \dfrac{1\wedge |f|}{|f|^{2}}\chi_{\{ f\neq 0 \} } -\left(\dfrac{1+\mathrm{sign}(1-|f|)}{2|f|} \right)\chi_{\{ f\neq 0 \} }  \right]  | f| \cdot | \nabla |f| \ |_{Q}^{2}   \, dx + E_{1} \geq 0,

%Now, define $ M(|f|):= \min \left\lbrace \alpha(|f|), \ \beta (|f|) \right\rbrace $. Note that $\alpha(|f|), \ \beta(|f|) \geq 0 $. Since $ \displaystyle \int_{\R^d}  \alpha(|f|) | f| \, | \nabla |f| \ |_{Q}^{2} dx \geq 0$, one obtains
%\begin{eqnarray*}
%\Re \A_{\tilde{\omega}} (f,(f-\p_{f})) & \geq  &   \int_{\R^d} M(|f|) \left[ ( \eta_{1}-\varepsilon) \sum_{i=1}^{m} | \nabla  f_{i} |^{2} - c_\varepsilon \sum_{i=1}^{m}  | f_i |^{2} \right]  \dx  + \, \tilde{\omega} \left\Vert \beta(|f|)^{\frac{1}{2}} f\right\Vert ^{2}_{2}\\
%& \geq  &  ( \eta_{1}-\varepsilon)  \| M(|f|)^{\frac{1}{2}} \, \nabla  f \|_{2}^{2}+( \tilde{\omega} -c_\varepsilon  ) \| M(|f|)^{\frac{1}{2}}\, f \|_{2}^{2}\\ 
 %&\geq & 0.
%\end{eqnarray*}

\end{proof}
%
%
%\begin{prop} 
%The semigroup $\{S_\omega(t)\}_{t\ge0}$ is $L^\infty$-contractive.
%\end{prop}
Hence, we have the following main result of this section. 
\begin{thm}
Let $1<p<\infty$ and assume Hypotheses (H1) and (H2). Then, $\mathcal{L}$ has a realization $L_p$ in $L^p(\R^d,\C^m)$ that generates an analytic $C_0$-semigroup $\{S_p(t)\}_{t\ge 0}$.
\end{thm}
\begin{proof}
Let $2<p<\infty$. Instead of considering $\min(\omega,\tilde{\omega})$, we assume $\omega>\tilde{\omega}$. In view of Theorem \ref{thm Generation L2} and Proposition \ref{Charac Ouhab}, the semigroup $\{S_{2}^{\omega}(t)\}_{t\ge 0}$ is analytic in $L^2(\R^d,\C^m)$ and $L^\infty$--contactive. Therefore, using the Riesz-Thorin interpolation Theorem, $\{S_{2}^{\omega}(t)\}_{t\ge 0}$ has a unique analytic bounded extension $\{S_p^\omega(t)\}_{t\ge 0}$ to $L^p(\R^d,\C^m)$. Moreover, for every  $f\in L^2(\R^d,\C^m)\cap L^\infty(\R^d,\C^m)$, one claims
\begin{align}\label{RiezTorin Strong Conti}\nonumber
 \|S_{p}^{\omega}(t)f-f\|_p &\le \|S_{2}^{\omega}(t)f-f\|_2^\theta\|S_{2}^{\omega}(t)f-f\|_\infty^{1-\theta}\\
 &\le 2^{1-\theta}\|f\|_\infty^{1-\theta}\|S_{2}^{\omega}(t)f-f\|_2^\theta,  
\end{align}
where $\theta=\frac{2}{p}$. Since by Theorem \ref{thm Generation L2}, the semigroup $\{S_{2}^{\omega}(t)\}_{t\ge 0}$ is strongly continuous in  $L^2(\R^d,\C^m)$, it follows directly from \eqref{RiezTorin Strong Conti} that $\{S^{\omega}_p(t)\}_{t\ge 0}$ is strongly continuous in $L^p(\R^d,\C^m)$.\\
For the case $1<p<2$, we argue by duality. Indeed, the adjoint semigroup $ \{S^{*}(t)\}_{t\ge 0} $ is associated to $\mathcal{L}^*$, the formal adjoint of $\mathcal{L}$, where
%we consider the adjoint operator $ \mathcal{L}^{*} $ on $L^2(\R^d,\C^m)$ defined by:
$$ \mathcal{L}^{*}f:=\mathrm{div}(Q\nabla f)-C^{*}\cdot\nabla f+ \mathrm{div}(F^{*}f)-V^{*}f  .$$ 
Since the coefficients of $ \mathcal{L}^{*} $ satisfy Hypotheses (H1) and (H2), similarly to $\mathcal{L}$, then $ \{S^{*}(t)\}_{t\ge 0} $ is an analytic $C_0$-semigroup in $L^2(\R^d,\C^m)$ which is quasi $L^\infty$-contractive. Consequently, $ \{S(t)\}_{t\ge 0} $ is quasi contractive in $L^1(\R^d,\C^m)$. So, the same interpolation arguments yield an extrapolation of $ \{S(t)\}_{t\ge 0} $ to a holomorpic $C_0$-semigroup in $L^p(\R^d,\C^m)$, for $1<p<2$.
%associated to the sesquilinear form $\A^{*} $, where $\A^{*} (f,g):=\overline{\A(g,f)}$, for all $f,g\in D(\A)$. Note that $\A^{*}  $ satisfies the same properties as the sesquilinear form $ \A $. Therefore, by Theorem \ref{thm Generation L2}, the realization $ L^{*} $ of $  \mathcal{L}^{*}$ in $ L^2(\R^d,\C^m)$ generates the analytic $ C_{0} $-semigroup $ \{S^{*}(t)\}_{t\ge 0} $. Hence, by $ L^\infty $--contractivity property, see Proposition \ref{Charac Ouhab}, we prove by duality that $\|S^{\omega}_{2}(t)f\|_1\le \|f\|_1$, for every $t>0$, and similarly to the first case, we obtain an extrapolation semigroup $\{S_{p}(t)\}_{t\ge 0}$ in $ L^p(\R^d,\C^m)$ which is strongly continuous and holomorphic.
\end{proof}

\begin{rmks}
a) The semigroups $\{S_{p}(t)\}_{t\ge 0}$, $1<p\le2$, can be extrapolated to a strongly continuous semigroup in $L^1(\R^d,\C^m)$. This follows, according to \cite{v92}, as a consequence of the consistency and the quasi-contractivity of $\{S_p(t)\}_{t\ge 0}$, $1<p\le2$.\\
b) If there exists a nonnegative locally bounded function $\mu:\R^d\to\R^+$ such that $\displaystyle\lim_{|x|\to\infty}\mu(x)=+\infty$ and
$$ \langle V_s(x)\xi,\xi\rangle \ge \mu(x)|\xi|^2,\qquad \forall x\in\R^d,\;\forall \xi\in\R^m .$$
 Then, for every $1<p<\infty$, $L_p$ has a compact resolvent and thus $\{S_p(t)\}_{t\ge0}$ is compact. The proof of this claim is identical to \cite[Proposition 4.3]{Mai}.
\end{rmks}

\section{Local elliptic regularity and maximal domain of $L_p$}\label{sec. maximal domain}

Since the coefficients of $\mathcal{L}$ are real, from now on, we consider vector--valued functions with real components. Thus, $L_p$ acts on $D(L_p)\subset L^p(\R^d,\R^m)$, for every $p\in(1,\infty)$ and its associated semigroup $\{S_p(t)\}_{t\ge0}$ acts on $L^p(\R^d,\R^m)$. Moreover, we assume that $C\equiv 0$ and thus
\begin{equation}\label{def 2 of L}
\mathcal{L}f=\mathrm{div}(Q\nabla f)-F\cdot\nabla f-Vf.
\end{equation}

Throughout this section, we use the notation $\Delta_Q:=\mathrm{div}(Q\nabla\cdot)$ and, in addition to Hypotheses (H1), we assume the following \\
\textbf{Hypotheses (H3)}:
%\begin{hyp}\label{hyp H2}
\begin{itemize}
\item $q_{ij}\in C_b^1(\R^d)$, for all $i,j\in\{1,\dots,d\}$.
%\item $F_{ij}\in W^{1,\infty}_\loc(\R^d)$, for all $i,j\in\{1,\dots,m\}$, and there exists $\gamma\in\R$ such that
%\begin{equation}\label{eq. div F>...}
%\langle \mathrm{div}(F)\xi,\xi\rangle:=\sum_{i,j=1}^{m}\mathrm{div}(F_{ij})\xi_i\xi_j \le \gamma|\xi|^2 
%\end{equation}
%for every $\xi\in\R^m$.
\item $v_{ij}\in L^\infty_\loc(\R^d)$, for all $i,j\in\{1,\dots,m\}$.
\end{itemize}
%\end{hyp}

\begin{rmk}
The assumption $C\equiv0$ is actually without loss of generalities. Indeed, %assuming, similarly to $F$, that $C_{ij}\in W^{1,\infty}_\loc(\R^d)$, for all $i,j\in\{1,\dots,m\}$, and $-\mathrm{div}(C)$ satisfies \eqref{eq. div F>...}, then, 
for every $f\in C_c^\infty(\R^d,\R^m)$, one has
\begin{align*}
\tilde{\mathcal{L}}f &:= \mathrm{div}(Q\nabla f)-F\cdot\nabla f+\mathrm{div}(Cf)-Vf\\
&=\mathrm{div}(Q\nabla f)-(F-C)\cdot\nabla f-(V-\mathrm{div}(C))f.
\end{align*}
Hence, $\tilde{\mathcal{L}}-\gamma$ has the same expression of \eqref{def 2 of L} and the matrices $Q$, $\tilde{F}:=F-C$ and $\tilde{V}:=V-\mathrm{div}(C)-\gamma I_m$ satisfy Hypotheses (H1) and (H2). 
\end{rmk}

\subsection{Local elliptic regularity}
Here we give a regularity result for weak solutions to systems of elliptic equations. The following theorem generalizes \cite[Theorem 7.1]{Agmon} to the vector valued case. 
\begin{thm}\label{thm ellip reg}
Let $p\in(1,\infty)$ and assume Hypotheses (H1)--(H3). Let $f$ and $g$ belong to $L^p_\loc(\R^d,\R^m)$ such that $\mathcal{L}f=g$ in the distribution sense. Then, $f\in W^{2,p}_{\loc}(\R^d,\R^m)$.
\end{thm}
\begin{proof}
Let $f=(f_1,\dots,f_m)$ and $g=(g_1,\dots,g_m)$ belong to $L^p_\loc(\R^d,\R^m)$ and assume that $\mathcal{L}f=g$ in the sense of distributions. Hence,
\begin{equation}\label{eq ellip reg}
 \Delta_Q f_i=g_i+\sum_{j=1}^{m}F_{ij}\cdot\nabla f_j+\sum_{j=1}^{m}v_{ij}f_j
\end{equation}
for each $i\in\{1,\dots,m\}$. Now, let $\varphi\in C_c^2(\R^d)$ and $i\in\{1,\dots,m\}$. A straightforward computation yields $$\Delta_Q(\varphi f_i)=\varphi\Delta_Q f_i+(Q\nabla\varphi)\cdot\nabla f_i+(\Delta_Q\varphi)f_i.$$
Then, by \eqref{eq ellip reg} one gets
\[ \Delta_Q(\varphi f_i)=\varphi g_i+\sum_{j=1}^{m}\varphi F_{ij}\cdot\nabla f_j
+(Q\nabla\varphi)\cdot\nabla f_i+\sum_{j=1}^{m}v_{ij}f_j\varphi+(\Delta_Q\varphi)f_i:=\tilde{g}_i.
\]
Actually, $\tilde{g}_i\in W^{-1,p}(\R^d):=(W^{1,p'}(\R^d))^{'}$. Indeed, since $g_i$ and $f_j$ belong to $L^p_\loc(\R^d,\R^m)$, then $\varphi g_i$, $(\Delta_Q\varphi)f_i$ and $v_{ij}f_j\varphi$ lie in $L^p(\R^d)$ and thus in $W^{-1,p}(\R^d)$, for every $j\in\{1,\dots,m\}$. On the other hand, for every $\psi\in C_c^\infty(\R^d)$, one has
\begin{align*}
|(\varphi F_{ij}\cdot\nabla f_j\,,\,\psi)| &=\left|-\int_{\R^d}f_j\mathrm{div}(\varphi\psi F_{ij})\, dx \right|\\
&= \left|\int_{\R^d}f_j\varphi\psi \mathrm{div}(F_{ij})\, dx+\int_{\R^d}f_j\psi \langle  F_{ij},\nabla\varphi \rangle\, dx \right. \\ &+\left. \int_{\R^d}f_j\varphi  \langle F_{ij}, \nabla\psi \rangle\, dx\right| \\
&\le \left(\|\mathrm{div}(F_{ij})\varphi f_j\|_p +\|\langle F_{ij}, \nabla\varphi \rangle f_j\|_p\right) \|\psi\|_{p'} \\&+\|F\|_\infty\|f_j\varphi\|_p \|\nabla\psi\|_{p'}\\
&\le \left( \|\mathrm{div}(F_{ij})\varphi f_j\|_p +\|\langle F_{ij},\nabla\varphi \rangle f_j\|_p+\|F\|_\infty\|f_j\varphi\|_p \right) \|\psi\|_{1,p'},
\end{align*}
which shows that $\varphi F_{ij}\cdot\nabla f_j\in W^{-1,p}(\R^d)$, for every $j\in\{1,\dots,m\}$. Similarly, we get the claim for $(Q\nabla\varphi)\cdot\nabla f_i$. Therefore, for all $\lambda>0$,
$$ (\Delta_Q-\lambda)(\varphi f_i)=\tilde{g}_i-\lambda\varphi f_i\in W^{-1,p}(\R^d). $$
Thus, according to \cite[Proposition 2.2]{BKR06}, $\varphi f_i\in W^{1,p}(\R^d)$ and this is true for every $\varphi\in C_c^\infty(\R^d)$, which implies that $f_i\in W^{1,p}_\loc(\R^d)$.

Now, coming back to \eqref{eq ellip reg}, one obtains $\Delta_Q f_i \in L^{p}_\loc(\R^d)$. We then conclude by \cite[Theorem 7.1]{Agmon} that $f_i$ belongs to $W^{2,p}_\loc(\R^d)$.
\end{proof}

\subsection{$L^p$-maximal domain}
The aim of this section is to coincide the domain $D(L_p)$ of the generator of $\{S_p(t)\}_{t\geq 0}$ with its maximal domain in $L^p(\R^d,\R^m)$. We start by showing that $C_c^\infty(\R^d;\C^m)\subset D(L_p)$.
\begin{lm}
Let $p\ge 1$ and assume Hypotheses (H1)--(H3). Then, $C_c^\infty(\R^d,\R^m)\subset D(L_p)$ and
 %\[ L_p f=\mathcal{L}f, \qquad \forall f \in C_c^{\infty}(\R^d,\R^m).\]
 $L_p f=\mathcal{L}f$, for all $f \in C_c^{\infty}(\R^d,\R^m)$.
 \end{lm}

\begin{proof}
Let $f\in C_c^\infty(\R^d,\C^m)$. One has $\mathcal{L}f\in L^2(\R^d,\R^m)$ and integrating by parts, one claims $\langle -\mathcal{L}f,g\rangle_2=\A(f,g)$, for all $g\in D(\A)$. Therefore, $f\in D(L_2)$ and $L_2 f=\mathcal{L}f$. Moreover, one has
\begin{equation}\label{eq all days}
 S_2(t)f-f=\int_0^t S_2(s)\mathcal{L}f \,ds, \qquad \forall t>0.
\end{equation}
Since $\mathcal{L}f\in L^p(\R^d,\R^m)$, for all $p\ge 1$, and by consistency of the semigroups $\{ S_p(t)\}_{t\ge0}$, $p\in [1,\infty)$, Equation \eqref{eq all days} holds true in $L^p(\R^d,\R^m)$, that is
$$ S_p(t)f-f=\int_0^t S_p(s)\mathcal{L}f \,ds, \qquad \forall t>0. $$
By consequence, $f\in D(L_p)$ and $L_p f=\mathcal{L}f$ for all $p\ge 1$.
\end{proof}

We next show that the space of test functions is a core for $L_p$, for $p\in(1,\infty)$. That is, $C_c^\infty(\R^d;\R^m)$ is dense in $D(L_p)$ by the graph norm.
\begin{prop}\label{Prop C_c is a core}
  Let $1<p<\infty$ and assume Hypotheses (H1)--(H3). Then, the set of test functions $C_c^\infty(\R^d;\R^m)$ is a core for $L_p$.
 \end{prop}
 \begin{proof}
  Fix $1<p<\infty$ and let $\lambda>\gamma$ be bigger enough so that it belongs to $\rho(L_p)$. It suffices to prove that $(\lambda-L_p)C_c^\infty(\R^d,\R^m)$ is dense in $L^p(\R^d,\R^m)$. For this purpose, let $f\in L^{p'}(\R^d;\R^m)$ be such that $\langle(\lambda-\mathcal{L})\varphi,f\rangle_{p,p'}=0$, for all $\varphi\in C_c^\infty(\R^d;\R^m)$. Then,
 \begin{equation}\label{ellip res equ}
 \lambda f-\Delta_Q f-F^*\cdot\nabla f+(V^*-\mathrm{div}(F)) f=0
 \end{equation}
 in the sense of distributions. By Theorem \ref{thm ellip reg}, one obtains $f_j\in W^{2,p'}_{\loc}(\R^d)$ for all $j\in\{1,\dots,m\}$. Then, \eqref{ellip res equ} holds true almost everywhere on $\R^d$.

Now, consider $\zeta\in C_c^\infty(\R^d)$ such that $\chi_{B(1)}\leq\zeta\leq\chi_{B(2)}$ and define $\zeta_n(\cdot)=\zeta(\cdot/n)$ for $n\in\N$. Assume $p'<2$ and multiply \eqref{ellip res equ} by $\zeta_n (|f|^2+\varepsilon^2)^{\frac{p'-2}{2}}f\in L^p(\R^d,\R^m)$ for $\varepsilon>0,\,n\in\N$. Integrating by parts, one obtains
 \begin{eqnarray*}
 0 &=& \lambda \int_{\R^d}\zeta_n (|f|^2+\varepsilon^2)^{\frac{p'-2}{2}}|f|^2 dx +\sum_{j=1}^{m}\int_{\R^d}\Big\langle\nabla f_j,\nabla\big(\zeta_n (|f|^2+\varepsilon^2)^{\frac{p'-2}{2}}f_j\big)\Big\rangle_{Q} dx \\
 & &+\sum_{i,j=1}^{m}  \int_{\R^d} \zeta_n (|f|^2+\varepsilon^2)^{\frac{p'-2}{2}}f_{i} \langle F_{ji},\nabla f_j \rangle dx \\
 & &  + \int_{\R^d}\zeta_n(|f|^2+\varepsilon^2)^{\frac{p'-2}{2}}\langle (V^*-\mathrm{div}(F^*)) f,f\rangle dx \\
 &\ge & (\lambda-\gamma)\int_{\R^d}\zeta_n(|f|^2+\varepsilon^2)^{\frac{p'-2}{2}}|f|^2\,dx+\sum_{j=1}^{m}\int_{\R^d}|\nabla f_j|_{Q}^2 \zeta_n(|f|^2+\varepsilon^2)^{\frac{p'-2}{2}}\,dx\\
 & & +\sum_{j=1}^{m}\int_{\R^d}\langle\nabla f_j,\nabla\zeta_n\rangle_{Q} (|f|^2+\varepsilon^2)^{\frac{p'-2}{2}}f_j\,dx 
  \\ & & - \| F \|_{\infty}\sum_{i,j=1}^{m}  \int_{\R^d}  \zeta_n (|f|^2+\varepsilon^2)^{\frac{p'-2}{2}} |f_{i} | \, | \nabla  f_j |dx \\
  & & +(p'-2)\sum_{j=1}^{m}\int_{\R^d}\langle\nabla f_j,\nabla|f|\rangle_{Q} f_j|f|\zeta_n(|f|^2+\varepsilon^2)^{\frac{p'-4}{2}}\,dx  \\
 &\ge &  (\lambda-\gamma)  \int_{\R^d}\zeta_n(|f|^2+\varepsilon^2)^{\frac{p'-2}{2}}|f|^2 dx+\int_{\R^d}\sum_{j=1}^{m}|\nabla f_j|^2_Q \zeta_n(|f|^2+\varepsilon^2)^{\frac{p'-2}{2}} dx \\
   & &+ \int_{\R^d}\langle Q\nabla|f|,\nabla\zeta_n\rangle (|f|^2+\varepsilon^2)^{\frac{p'-2}{2}}|f| dx \\
 & & - \delta \sum_{j=1}^{m}  \int_{\R^d}  \zeta_n (|f|^2+\varepsilon^2)^{\frac{p'-2}{2}}  | \nabla  f_j |^{2}_Q dx - C_\delta  \int_{\R^d}  \zeta_n (|f|^2+\varepsilon^2)^{\frac{p'-2}{2}} |f |^{2}dx \\ && +(p'-2)\int_{\R^d}|\nabla|f||^2\zeta_n |f|^2(|f|^2+\varepsilon^2)^{\frac{p'-4}{2}}\,dx  \\
%%%%%%%%%%%%%%%%%%%%%%%%%%%%%%%%%%%%%%%%%%%%%%%%%
%%%%%%%%%%%%%%%%%%%%%%%%%%%%%%%%%%%%%%%%%%%%%%%%%
 \end{eqnarray*}
%  \begin{eqnarray*}
%  -\sum_{i,j=1}^{m}  \int_{\R^d}  \mathrm{div}\left( \zeta_n (|u|^2+\varepsilon^2)^{\frac{p'-2}{2}}u_{i}F_{ji}\right) \cdot u_j dx = 
%   \end{eqnarray*}
%%%%%%%%%%%%%%%%%%%%%%%%%%%%%%%%%%%%%%%%%%%%%%%%%
%%%%%%%%%%%%%%%%%%%%%%%%%%%%%%%%%%%%%%%%%%%%%%%%%
for all $ \delta>0$ and some $C_{\delta}  > 0$. %By  a special choice of $ \delta, \ C_{\delta} $, we obtain $ \tilde{C}:= \tilde{C}(\eta_{1}, \eta_{2} ) \geq 0$ and 
Moreover, according to \cite[Lemma 2.4]{Mai-Rha}, one has
$$ |\nabla|f||_Q^2 \le \sum_{j=1}^{m}|\nabla f_j|^2_Q  .$$
So that, choosing $\delta=\delta_p<p'-1$ and $\lambda>\gamma+C_{\delta_p}$, one gets
 \begin{eqnarray*}
 0 &\ge & (\lambda -\gamma -C_{\delta_p} )  \int_{\R^d}\zeta_n(|f|^2+\varepsilon^2)^{\frac{p'-2}{2}}|f|^2\, dx \\ 
 &&+  \int_{\R^d}\langle Q\nabla|f|,\nabla\zeta_n\rangle  (|f|^2+\varepsilon^2)^{\frac{p'-2}{2}}|f| \,dx\\
  & & +(p'-1-\delta)\int_{\R^d}|\nabla|f||^2\zeta_n |f|^2(|f|^2+\varepsilon^2)^{\frac{p'-4}{2}}\,dx\\
  &\ge & (\lambda-\gamma -C_{\delta_p} )   \int_{\R^d}\zeta_n(|f|^2+\varepsilon^2)^{\frac{p'-2}{2}}|f|^2\, dx+\frac{1}{p'}\int_{\R^d}\langle Q\nabla((|f|^2+\varepsilon^2)^{\frac{p'}{2}})),\nabla\zeta_n\rangle \,dx\\
  &=&(\lambda -\gamma -C_{\delta_p} )  \int_{\R^d}\zeta_n(|f|^2+\varepsilon^2)^{\frac{p'-2}{2}}|f|^2\, dx-\frac{1}{p'}\int_{\R^d} \Delta_Q \zeta_n (|f|^2+\varepsilon^2)^{\frac{p'}{2}})\,dx.
 \end{eqnarray*}
 Upon $\varepsilon\to 0$, one obtains
 $$(\lambda -\gamma -C_{\delta_p})  \int_{\R^d}\zeta_n |f|^{p'}\,dx-\frac{1}{p'}\int_{\R^d} \Delta_Q\zeta_n|f|^{p'}\,dx\le 0.$$
A straightforward computation yields
\[ \Delta\zeta_n=\frac{1}{n}\sum_{i,j=1}^{m}\partial_i q_{ij}\partial_j\zeta(\cdot/n)+\frac{1}{n^2}\sum_{i,j=1}^{m}q_{ij}\partial_{ij}\zeta(\cdot/n). \]
So that $\|\Delta_Q \zeta_n\|_\infty$ tends to $0$ as $n\to\infty$. Therefore, upon $n \to \infty$, one claims
 $$\int_{\R^d}|f|^{p'}\, dx\le 0.$$
 Hence, $f=0$.\\
 On the other hand, if $p'\geq 2$, multiplying \eqref{ellip res equ} by $\zeta_n |f|^{p'-2}f$, in a similar way, one gets
 %  \begin{eqnarray*}
 % \sum_{i,j=1}^{m}  \int_{\R^d}  \mathrm{div}\left( \zeta_n |u|^{p-2}u_{i}C_{ij} u_j\right) dx = -\sum_{i,j=1}^{m}  \int_{\R^d}  |u|^{p-2}u_{i} \left(  C_{ij} \cdot \nabla \zeta_n \right) u_jdx
 %    \end{eqnarray*}
 \begin{eqnarray*}
 	0&=&\lambda\int_{\R^d}\zeta_n|f|^{p'} dx+\int_{\R^d}\sum_{j=1}^{m}\langle Q\nabla f_j,\nabla(|f|^{p'-2}f_j\zeta_n)\rangle dx \\
 	& &  + \sum_{i,j=1}^{m}  \int_{\R^d} \zeta_n |f|^{p'-2}f_{i}\langle F_{ji}, \nabla  f_j \rangle dx \\
 	& & +\int_{\R^d}\langle (V^*-\mathrm{div}(F^*))f,f\rangle |f|^{p'-2}\zeta_n dx\\
 	&\geq & (\lambda-\gamma)\int_{\R^d}\zeta_n|f|^{p'} dx+\eta_{1}\sum_{j=1}^{m} \int_{\R^d}|f|^{p'-2}\zeta_n| \nabla f_j|^{2} dx \\
 	& & +\eta_{1}\int_{\R^d}\sum_{j=1}^{m}|f|^{p'-2}f_j\langle \nabla f_j,\nabla \zeta_n\rangle dx \\
 	& & +\eta_{1}(p'-2) \int_{\R^d}|f|^{p'-2}\zeta_n|  \nabla |f| |^{2} dx \\ && - \|F\|_{\infty}\sum_{i,j=1}^{m}  \int_{\R^d} \zeta_n |f|^{p'-2} |f_{i}| | \nabla f_j | dx\\
 	&\geq & (\lambda-\gamma)\int_{\R^d}\zeta_n|f|^{p'} dx+\eta_{1}\sum_{j=1}^{m} \int_{\R^d}|f|^{p'-2}\zeta_n| \nabla f_j|^{2} dx \\
 	& & +\eta_{1}\int_{\R^d}|f|^{p'-1}\langle \nabla| f|,\nabla \zeta_n\rangle dx -  C_\delta \int_{\R^d} \zeta_n |f |^{p'}   dx \\
 	& & +\eta_{1}(p'-2) \int_{\R^d}|f|^{p'-2}\zeta_n|  \nabla |f| |^{2} dx  \\
 	&&-\delta \int_{\R^d} \zeta_n  |f|^{p'-2} | \nabla |f| |^{2} dx\\
 	&\geq &(\lambda -\gamma-C_\delta )\int_{\R^d}\zeta_n|f|^{p'} dx+\eta_{1}\int_{\R^d}|f|^{p'-1}\langle \nabla |f|,\nabla \zeta_n\rangle dx \\ 
 	& &+ \eta_{1}(p'-1) \int_{\R^d}|f|^{p'-2}\zeta_n|  \nabla |f| |^{2} dx\\
 	&\geq &(\lambda -\gamma-C_\delta )\int_{\R^d}\zeta_n|f|^{p'} dx+\frac{\eta_{1}}{p'}\int_{\R^d}\langle \nabla \zeta_n, \nabla |f|^{p'}\rangle dx\\
 	%\int_{\R^d}|u(x)|^{p-2}\zeta_n^2(x)\sum_{j=1}^{m}\langle Q(x)\nabla u_j(x),\nabla u_j(x)\rangle dx\\
 	% &+& 2\int_{\R^d}|u(x)|^{p-2}\sum_{j=1}^{m}u_j(x)\zeta_n(x)\langle Q(x)\nabla f_j(x),\nabla \zeta_n(x)\rangle dx\\
 	&\geq & (\lambda -\gamma-C_\delta )\int_{\R^d}\zeta_n|f|^{p'} dx-\frac{\eta_{1}}{p'}\int_{\R^d}\Delta_Q\zeta_n |f|^{p'}dx.
 	%\sum_{j=0}^{m}\langle Q(x)\nabla \zeta_n (x),\nabla \zeta_n (x)\rangle u_j^2(x)\\
 	%&\geq & \lambda\int_{\R^d}\zeta_n^2(x)|u(x)|^p dx-\frac{\eta_2\|\nabla\zeta\|_\infty^2}{n}\int_{\R^d}|u(x)|^p dx.
 \end{eqnarray*}
 It thus follows that $f=0$ by letting $n$ tends to $\infty$.
 
 %Letting $n$ tends to $\infty$, one obtains $0\ge \int_{\R^d}(|u|^2+\varepsilon^2)^{\frac{p'-2}{2}}|u|^2\, dx$ and, letting $\varepsilon$ tends to $0$, one gets $0\ge\|u\|_{p'}^{p'}$. Therefore, $u=0$.\\
 %In the case when $p'>2$, one multiplies in \eqref{ellip res equ} by $\zeta_n|f|^{p'-2}f$ and argues in a similar way.
% (ii) This is an immediate consequence of (i) and \cite[Corollary III-5.8]{Nagel}.
\end{proof}

We show in the next that the domain $D(L_p)$ is equal to the $L^p$-maximal domain of $\mathcal{L}$.

\begin{prop}\label{Coincidence of domains}
Let $1<p<\infty$ and assume Hypotheses (H1)--(H3). Then
$$D(L_p)=\{f\in L^p(\R^d,\R^m)\cap W^{2,p}_{\loc}(\R^d;\R^m) : \mathcal{L}f\in L^p(\R^d;\R^m)\}:=D_{p,\max}(\mathcal{L}).$$
\end{prop}
\begin{proof}
We first show that $D(L_p)\subseteq D_{p,\max}(\mathcal{L})$. Let $f\in D(L_p)$ and $(f_n)_n\subset C_c^\infty(\R^d,\R^m)$ such that $f_n\to f$ and $\mathcal{L}f_n\to L _p f$ in $L^p(\R^d,\R^m)$. Let $\Omega$ be a bounded domain of $\R^d$ and $\phi\in C_c^2(\Omega)$. Consider, on $\Omega$, the differential operator $$ \Lambda=\mathcal{L}-2\langle Q\nabla\phi
,\nabla\cdot\rangle. $$
A straightforward computation yields
$$ \Lambda(\phi f_n)=\phi\mathcal{L}f_n+(\Delta_Q\phi-2\langle Q\nabla\phi
,\nabla\phi\rangle)f_n+\sum_{j=1}^{m}\langle F_{ij},\nabla\phi\rangle \langle f_n,e_j\rangle. $$
Thus, $(\Lambda(\phi f_n))_{n}$ converges in $L^p(\Omega,\R^m)$. Taking into the account that $\Lambda$ is an elliptic operator with bounded coefficients on $\Omega$, thus the domain of $\Lambda$, with Dirichlet boundary condition, coincides with $W^{2,p}(\Omega,\R^m)\cap W^{1,p}_0(\Omega,\R^m)$. In particular, $(\phi f_n)_{n}$ converges in $W^{2,p}(\Omega,\R^m)$, which implies that $\phi f\in W^{2,p}(\Omega,\R^m)$. Now, the arbitrariness of $\Omega$ and $\phi$ yields $f\in W^{2,p}_{\loc}(\R^d,\R^m)$. Furthermore, $(\mathcal{L}f_n)_n$ converges locally in $L^p(\R^d,\R^m)$ to $\mathcal{L}f$ and by pointwise convergence of subsequences, one claims $L_p f=\mathcal{L}f$.
%for every $\psi\in C_{c}^{\infty}(\R^d,\R^m)$ and $n\in\N$, one has
%\begin{align*}
%\langle-\mathcal{L}f_n,\psi\rangle &= \sum_{i=1}^{m}\int_{\R^d}\langle Q\nabla f_n^{(i)} ,\nabla \psi_i \rangle\, dx\\
%&+ \sum_{i,j=1}^{m}\int_{\R^d}\left(F_{ij}\cdot\nabla f_n^{(j)}\right) \psi_i\, dx+\int_{\R^d}\langle V f_n,\psi\rangle\, dx,
%\end{align*}
%From $V\in L^\infty_\loc(\R^d,\C^m)$, we deduce that $V u_n\to Vu$ in $L^p_\loc(\R^d,\C^m)$. Consequently, $$\Delta_Q u=A_pu+Vu=\lim_{n\to\infty}\mathcal{A}u_n+Vu_n \in L^p_\loc(\R^d,\C^m).$$ So, by local elliptic regularity, we obtain $u\in W^{2,p}_\loc(\R^d,\C^m)$. Hence, $\mathcal{A}u=A_p u$ belongs to $L^p(\R^d,\C^m)$, which shows that $u\in D_{p,\max}(\mathcal{A})$.

In order to prove the other inclusion it suffices to show that $\lambda-\mathcal{L}$ is one to one on $D_{p,\max}(\mathcal{L})$, for some $\lambda>0$. Indeed, this implies that $\lambda\in\rho(\mathcal{L}_{p,\max})\cap \rho(L_p)$, where $\mathcal{L}_{p,\max}$ is the realization of $\mathcal{L}$ on $D_{p,\max}(\mathcal{L})$. Since $D_{p,\max}(\mathcal{L})\subset D(L_p)$, thus $L_p=\mathcal{L}_{p,\max}$.
Now, let $f\in D_{p,\max}(\mathcal{L})$ be such that $(\lambda-\mathcal{L})f=0$. Arguing similarly as in the proof of Proposition \ref{Prop C_c is a core}, one obtains $f=0$ and this ends the proof.
%Consider $\zeta\in C_c^\infty(\R^d)$ such that $\chi_{B(1)}\leq\zeta\leq\chi_{B(2)}$ and define $\zeta_n(\cdot)=\zeta(\cdot/n)$ for all $n\in\N$.

 %The case $p<2$ can be obtained similarly, by multiplying the equation $(\lambda -\mathcal{A})f=0$ by $\zeta_n(|f|^2+\varepsilon)^{\frac{p-2}{2}}f$, $\varepsilon>0$, instead of $\zeta_n|f|^{p-2}f$.
 %In order to prove that $C_c^\infty(\R^d,\R^m)$ is a core for $L_p$ it suffices to show that $(\lambda-L_p)C_c^\infty(\R^d,\R^m)$ is dense in $L^p(\R^d,\R^m)$ for some $\lambda>0$. For this purpos let $f\in L^{p'}(\R^d,\R^m)$ such that $\langle (\lambda-L_p)\varphi, f\rangle=0$ for all $\varphi\in C_c^\infty(\R^d,\R^m)$. Hence,
 %\[ \int_{\R^d}\langle\varphi(x),(\lambda+V^*(x)f(x)\rangle dx=\int_{\R^d}\langle div(q\nabla\varphi)(x),f(x)\rangle dx,\qquad\forall\varphi\in C_c^\infty(\R^d,\R^m). \]
 %Standard elliptic regularity yields $f\in W^{2,p'}_{loc}(\R^d,\R^m)$ and
 %\[ \int_{\R^d}\langle\varphi(x),(\lambda+V^*(x)f(x)\rangle dx=\int_{\R^d}\langle \varphi(x),div(q\nabla f)(x)\rangle dx,\qquad\forall\varphi\in C_c^\infty(\R^d,\R^m). \]
 %Then, $(\lambda+V^*)f=div(q\nabla f)$ a.e. Hence, $f\in D(L_p^*)$ and $(\lambda-(L_p^*))f=0$ thus $f=0$. This ends the proof.
\end{proof}
\begin{rmk}
It is relevant to have $D(L_p)\subset W^{2,p}(\R^d,\R^m)$, for $1<p<\infty$, which is equivalent to the coincidence of domains $D(L_p)=W^{2,p}(\R^d,\R^m)\cap D(V_p)$, where $D(V_p)$ refers to the maximal domain of multiplication by $V$ in $L^p(\R^d,\R^m)$. Actually, in \cite[Section 3]{Mai-Rha}, it has been shown the following
$$  \|f\|_{2,p}+\|Vf\|_p \le C( \|\Delta_Q f-Vf\|_p +\|f\|_p)  $$
for all $f\in W^{2,p}(\R^d,\R^m)\cap D(V_p)$, provided that $V=\hat{V}+v I_m$, with $0\le v \in W^{1,p}_\loc(\R^d)$ such that $|\nabla v|\le C v$ and $\hat{V}$ satisfies 
$$ \sup_{1\le j\le m}\|(\partial_j \hat{V}) \hat{V}^{-\gamma}\|_\infty<\infty $$
for some $\gamma\in[0,1/2)$. Now, taking into the account, the Landau's inequality
$$  \|\nabla f\|_p\le \varepsilon \|\Delta_Q f\|_p+M_\varepsilon \|f\|_p ,$$
for every $\varepsilon>0$, one claims
$$  \|f\|_{2,p}+\|Vf\|_p \le C'( \|L_p f\|_p +\|f\|_p) .  $$
Therefore, $D(L_p)=W^{2,p}(\R^d,\R^m)\cap D(V_p)$. 
\end{rmk}

\section{Positivity}\label{sect 5}
In this section we characterize the positivity of the semigroup $\{S_{p}(t)\}_{t\ge0}$ for $ 1<p<\infty $. Since the family of semigroups $\{S_{p}(t)\}_{t\ge0}$, $p\in[1,\infty)$, is consistent, i. e., $ S_p(t)f=S_q(t)f $, for every $t\ge0$, $1\le p,q<\infty$ and all $f\in L^p(\R^d,\R^m)\cap L^q(\R^d,\R^m)$, it suffices to characterize the positivity of $\{S_{2}(t)\}_{t\ge0} $. For this purpose, we endow $\R^m$ with the usual partial order: $x\ge y$ if and only if, $x_i\ge y_i$, for all $i\in\{1,\dots,m\}$.  As in Section \ref{sec. maximal domain}, we assume that $C\equiv0$. By positivity of $\{S_{2}(t)\}_{t\ge0}$ we mean $S_{2}(t)f\ge0$ a.e., for every $t\ge0$ and all $f\in L^2(\R^d,\R^m)$ such that $f\ge0$ a.e.

We apply the Ouhabaz' criterion for invariance of closed convex subsets by semigroups, c.f. \cite[Theorem 3]{Ouhabaz 99} and \cite{Ouhabaz 96}. We then get the following result

\begin{thm}
Assume Hypotheses (H1). Then, the semigroup $\{S_{2}(t)\}_{t\ge0}$ is positive, if and only if, $F_{ij}=0$ and $v_{ij}\le0$ almost everywhere and for every $i\neq j\in\{1,\cdots,m\}$.
\end{thm}
\begin{proof}
Let $\mathcal{C}=\{f\in L^2(\R^d,\R^m): f\ge0\; \mathrm{a.e.}\}$ and $\mathcal{P}_+ f=f^+=(f_i^+)_{1\le i\le m}$, where $f_{i}^+=\max(0,f_i)$. Then, $\mathcal{C}$ is a closed convex subset of $L^2(\R^d,\R^m)$ and $\mathcal{P}_+$ is the corresponding projection. Now, let $\omega\ge0$ such that $\A_\omega$ is accretive. According to \cite[Theorem 3 (iii)]{Ouhabaz 99}, $\{S_{2}(t)\}_{t\ge0}$ is positive if, and only if, the form $\A$ satisfies the following
\begin{itemize}
\item $f\in D(\A)$ implies $f^+\in D(\A)$,
\item $\A_\omega(f^+,f^-)\le0$, for all $f\in D(\A)$, where $f^-=f-f^+$.
\end{itemize}
Now, assume that $\{S_{2}(t)\}_{t\ge0}$ is positive. Let $i\neq j\in\{1,\dots,m\}$, $n\in\N$ and $0\le\varphi\in C_c^\infty(\R^d)$. Set $f=\zeta_n e_i-\varphi e_j$. One has
\begin{align*}
0\ge \A_\omega(f^+,f^-)=\frac{1}{n}\int_{\R^d}\langle F_{ij},\nabla\zeta(\cdot/n)\rangle\,\varphi\, dx+\int_{\R^d}v_{ij}\zeta_n\varphi \,dx.
\end{align*}
Letting $n\to\infty$, by dominated convergence theorem, one gets $\int_{\R^d}v_{ij}\varphi\, dx\le0$ for every $0\le\varphi\in C_c^\infty(\R^d)$, which implies that $v_{ij}\le0$ almost everywhere. On the other hand, considering, for every $n\in\N$, 
$$g(x)=g^{(k,n)}(x):=\exp(nx_k)\varphi(x)e_i-\exp(-nx_k)\varphi(x) e_j,$$
 where $x_k$ is the k-th component of $x\in\R^d$, for every $k\in\{1,\cdots,d\}$. Then,
$$ \nabla g^+_i=n\exp((nx_k)\varphi e_k+\exp(nx_k)\nabla\varphi.$$
Therefore,
\begin{align*}
0\ge \frac{1}{n}\A_\omega(g^+,g^-)&=\int_{\R^d} F_{ij}^{(k)}\varphi^2\,dx+\frac{1}{n}\int_{\R^d}\langle F_{ij},\nabla\varphi\rangle\varphi\,dx\\
&\quad +\frac{1}{n}\int_{\R^d}v_{ij}\varphi^2 \,dx,
\end{align*}
where $F_{ij}^{(k)}$ indicates the $k$-th component of $F_{ij}$. So, by letting $n\to\infty$, one deduces that $F_{ij}^{(k)}\le0$ almost everywhere and for each $k\in\{1,\cdots,d\}$. In a similar way, one gets $F_{ij}^{(k)}\ge0$ a.e. by considering $\tilde{g} $ instead of $g$, where
$$\tilde{g}(x)=\tilde{g}^{(k,n)}(x):=\exp(-nx_k)\varphi(x)e_i-\exp(nx_k)\varphi(x) e_j.$$  
So that $F_{ij}=0$ almost everywhere.
%\begin{equation}\label{eq pf pos}
 %\int_{\R^d}\frac{\exp(\arctan(nx_k))}{1+x_k^2} F_{ij}^{(k)}(\chi_{\{x_k>0\}}-\chi_{\{x_k<0\}})\varphi\,dx\le0.
%\end{equation}
%for every $0\le\varphi\in C_c^\infty(\R^d)$. Hence $F_{ij}^{(k)}\le0$ a.e.
%Furthermore, taking $\exp(-\arctan(ne^{-x_k}))$ instead of $\exp(\arctan(nx_k))$ in the expression of $g$, one gets similarly 
%$$ \int_{\R^d}\frac{\exp(\arctan(nx_k))}{1+x_k^2} F_{ij}^{(k)}(\chi_{\{x_k>0\}}-\chi_{\{x_k<0\}})\varphi\,dx\ge0.$$
%Then \eqref{eq pf pos} becomes an equality. In particular,
%\[ \int_{\{x_k>0\}}\frac{\exp(\arctan(nx_k))}{1+x_k^2} F_{ij}^{(k)}\varphi\,dx=0\]
%for every $0\le\varphi\in C_c^\infty(\R^d)$ with support contained in $\{x\in\R^d: x_k>0\}$. Therefore, $F_{ij}^{(k)}=0$ almost everywhere on $\{x\in\R^d: x_k\ge0\}$. By the same way one can consider nonnegative test functions with support in $\{x\in\R^d: x_k<0\}$ and obtains $F_{ij}^{(k)}=0$ for $x_k\le0$ and thus $F_{ij}=0$ a.e.\\

Conversely, assume $F_{ij}=0$ and $v_{ij}\le0$ for all $i\neq j\in\{1,\dots,m\}$. Let $f\in D(\A)$, then, by \cite[Theorem 4.2]{Mai}, one gets $f^+\in D(a)$. Furthermore, it follows, by \cite[Theorem 7.9]{Gilb-Tru}, that $\nabla f^+_i=\chi_{\{f_i>0\}}\nabla f_i$ and $\nabla f^-_i=\chi_{\{f_i<0\}}\nabla f_i$. Let us now prove that $\A_\omega(f^+,f^-)\le 0$. One has
\begin{eqnarray*}
 \A_\omega(f^+,f^-) &=&\displaystyle\sum_{i=1}^{m}\int_{\R^d}\langle Q\nabla f^{+}_i ,\nabla f^{-}_i \rangle\, dx +\sum_{i=1}^{m}\int_{\R^d}\langle F_{ii}, \nabla f^{+}_i\rangle  f^{-}_i\, dx\\
 & &+\sum_{i,j=1}^{m}\int_{\R^d} v_{i,j} f_{i}^{+}f_{j}^{-} \, dx + \omega \langle f^{+},f^{-}  \rangle_{2}\\
 &=&  \sum_{i\neq j}^{m}\int_{\R^d} v_{i,j} f_{i}^{+}f_{j}^{-}  \, dx 
 \\ & \leq& 0 .
\end{eqnarray*}

This ends the proof.
\end{proof}

\end{document}